\documentclass[10pt]{amsart}

\usepackage{amsmath}
\usepackage{amsfonts}
\usepackage{amssymb}
\usepackage{amsthm}

\newtheorem{thrm}{Theorem}[section]
\newtheorem{lem}[thrm]{Lemma}
\newtheorem{cor}[thrm]{Corollary}
\newtheorem{main}{Theorem}

\begin{document}

\title{On the local Frattini subgroup of a finite group}

\author[D.G.\ Costanzo]{David G.\ Costanzo}
\address{School of Mathematical and Statistical Sciences, Clemson University, Clemson, SC 29634, USA}
\email{davidgcostanzo@gmail.com}


\maketitle

\begin{abstract}
Let $G$ be a finite group and let $\pi(G)$ be the set of prime divisors of $|G|$.
For each $p\in \pi(G)$, choose a Sylow $p$-subgroup $G_p$ of $G$.
The subgroup $\Psi(G)= \prod_{p\in \pi(G)} \Phi(G_p)_G$ is called the \textit{local Frattini subgroup} of $G$.
Here, for $J \leqslant G$, the symbol $J_G$ denotes the core of $J$ in $G$.
The subgroup $\Psi(G)$ is a characteristic subgroup of $G$ that is contained in $\Phi(G)$, the Frattini subgroup of $G$.
We prove that $\Psi(G)$ satisfies many properties similar to $\Phi(G)$.
Let $J$ be a normal $p$-subgroup of $G$ with $J \leqslant \Phi(G)$, and let $S$ be a Sylow $p$-subgroup of $G$.
John S. Rose proved that if $J \leqslant \mathbf{Z}(S)$, then $J \leqslant \Phi(S)$.
So, in fact, $J \leqslant \Phi(S)_G \leqslant \Psi(G)$.
Kaplan, Lev, Mann, and Solomon found various other conditions that would ensure $J$ is contained in $\Phi(S)$.
In this article, we add a few more conditions to this list.
\end{abstract}

\maketitle

\section{Introduction}

All groups in this article are finite.
Let $G$ be a group.
Recall that the \textit{Frattini subgroup} of $G$, which is usually denoted by $\Phi(G)$, is the intersection of all of the maximal subgroups of $G$.
(If $G = 1$, then we simply define $\Phi(G)=1$.)
Many variations of the Frattini subgroup of a group exist in the literature.
In most cases, some particular family of maximal subgroups is specified and then the corresponding intersection is studied.

On the other hand, in this article, we throw our attention at a certain characteristic subgroup found \textit{inside} the Frattini subgroup.
Let $G$ be a group and let $J$ be a subgroup of $G$.
Recall that the \textit{core} of $J$ in $G$, which is denoted by $J_G$, is the largest normal subgroup of $G$ that is contained in $J$.
Equivalently, $J_G = \bigcap_{g \in G} J^g$, the intersection of all conjugates of $J$ in $G$.
As usual, for a group $G$, we let $\pi(G)$ be the set of prime divisors of $|G|$.
Now, for each $p \in \pi(G)$, choose a Sylow $p$-subgroup, say $G_p$, of $G$.
Then, we define the \textit{local Frattini subgroup} of $G$ by \[ \Psi(G) = \prod_{p \in \pi(G)} \Phi(G_p)_G.\]
This subgroup, which is characteristic in $G$, was introduced by G. Kaplan, A. Lev, A. Mann, and R. Solomon in \cite{adequatefieldextensions}.
If $K$ is a normal subgroup of $G$ and if $K \leqslant \Phi(H)$ for some subgroup $H$ of $G$, then $K \leqslant \Phi(G)$.
This basic fact is listed below as Lemma \ref{normal,containedinafrattini}.
In particular, notice that  $\Phi(G_p)_G$ is, therefore, contained in $\Phi(G)$ for each $p \in \pi(G)$.
Consequently, $\Psi(G)$ is indeed a subgroup of $\Phi(G)$.
We mention also that if $P$ and $S$ are Sylow $p$-subgroups of $G$, then $\Phi(P)_G = \Phi(S)_G$.
So, the definition of $\Psi(G)$ really does not depend on the chosen Sylow subgroups of $G$.

There is an equivalent characterization of the local Frattini subgroup of a group which we now mention.
Let $G$ be a group.
If $H$ is a subgroup of $G$, then write $\pi(G: H)$ for the set of prime divisors of $|G:H|$.
So, in other words, $\pi(G:H) = \pi(|G:H|)$, where, as usual, $\pi(n)$ is the set of prime divisors of the positive integer $n$.
Following \cite{adequatefieldextensions}, a normal subgroup $K$ of $G$ is said to be a \textit{strictly Frattini subgroup} of $G$ if $\pi(G:K) = \pi(G: KX)$ for each subgroup $X$ of $G$.
We point out immediately that a strictly Frattini subgroup of $G$ is contained in $\Phi(G)$.
It turns out that the product of any collection of strictly Frattini subgroups is again a strictly Frattini subgroup, and so each group $G$ has a unique largest strictly Frattini subgroup.
This unique largest strictly Frattini subgroup is, in fact, $\Psi(G)$.
So, said differently, $\Psi(G)$ is the subgroup generated by the strictly Frattini subgroups of $G$.

We are able to prove the somewhat peculiar fact that any normal subgroup of $G$ that is contained in $\Psi(G)$ is actually strictly Frattini in $G$.

\begin{main}
If $G$ is a group and if $J$ is a normal subgroup of $G$ with $J \leqslant \Psi(G)$, then $J$ is a strictly Frattini subgroup of $G$.
\end{main}

Let $G$ be a group.
A well-known result is that $G' \cap \mathbf{Z}_\infty(G) \leqslant \Phi(G)$, where $G'$ is the derived subgroup of $G$ and where $\mathbf{Z}_\infty(G)$ is the terminal term of the upper central series of $G$.
We prove this result below as Lemma \ref{G'intersecthypercenterinfrattini}.
The characteristic subgroup $\mathbf{Z}_\infty(G)$, often called the \textit{hypercenter} of $G$, has a few interesting characterizations which can be found in \cite{lewis}.
We prove that the intersection $G' \cap \mathbf{Z}_\infty(G)$ is, in fact, contained in the local Frattini subgroup $\Psi(G)$.

\begin{main}
For each group $G$, the containment $G' \cap \mathbf{Z}_\infty(G) \leqslant \Psi(G)$ holds.
\end{main}

Again, let $G$ be a group.
Recall that the \textit{Fitting subgroup} of $G$ is the unique largest normal nilpotent subgroup of $G$ and is usually denoted by $\mathbf{F}(G)$.
For any group $G$, the section $\mathbf{F}(G)/\Phi(G)$ is abelian.
We are able to prove that the larger section $\mathbf{F}(G)/\Psi(G)$ is also abelian.

\begin{main}\label{sectionabelian}
If $G$ is a group, then the factor group $\mathbf{F}(G)/\Psi(G)$ is abelian.
\end{main}

Of course, from the previous theorem, we get the familiar fact that $\mathbf{F}(G)/\Phi(G)$ is abelian for each group $G$ as a corollary.

If a group $G$ is the direct product of its subgroups $J$ and $R$, then $\Phi(G)$ is the direct product of $\Phi(J)$ and $\Phi(R)$.
In other words, the Frattini subgroup of a direct product is the direct product of the Frattini subgroups of the direct factors.
We can prove that the analogous result holds for the local Frattini subgroup.

\begin{main}
If $G$ is a group and if $G = J \times R$, then $\Psi(G) = \Psi(J) \times \Psi(R)$.
\end{main}

Let $G$ be a group and let $J$ be a normal subgroup of $G$.
Assume that $J \leqslant \Phi(G)$.
Such a subgroup is usually called a \textit{Frattini normal subgroup} of $G$.
Now, suppose that $J \leqslant R \leqslant G$.
In \cite{rosefrattini}, John S. Rose found various conditions that would imply that $J \leqslant \Phi(R)$.
One of the main result in this direction is what we call Rose's Theorem and is listed in the next section as Theorem \ref{Rose'stheorem}.
We are able to add a few more results in this direction.

\begin{main}
Let $G$ be a group, let $p$ be a prime, and let $J$ be a normal $p$-subgroup of $G$ with $J \leqslant \Phi(G)$.
Let $S$ be a Sylow $p$-subgroup of $G$.
If $[S,J]$ is a normal subgroup of $G$, then $J \leqslant \Phi(S)$.
\end{main}

Notice that since $J$ is a normal subgroup of $G$, we actually have that $J \leqslant \Phi(S)_G \leqslant \Psi(G)$.

The next theorem asserts that if $J$ is a Frattini normal $p$-subgroup of $G$ that is contained in the hypercenter of $G$, then $J$ is also a Frattini normal subgroup of $S$.

\begin{main}
Let $G$ be a group, let $p$ be a prime, and let $J$ be a normal $p$-subgroup of $G$.
Let $S$ be a Sylow $p$-subgroup of $G$.
If $J \leqslant (\Phi(G) \cap \mathbf{Z}_\infty(G))$, then $J \leqslant \Phi(S)$.
\end{main}

The notation is standard.
Any of the classic textbooks like \cite{isaacs}, \cite{rosetextbook}, or \cite{scott} provide more than adequate background.
Moreover, in an effort to make the article mostly self-contained, we have provided proofs of some of the basic known results in the next section.

\section{Preliminaries}

Our goal in this section is to gather up the basic tools needed to prove the results that were stated in the introduction.
We first record some of the most useful properties of the Frattini subgroup of a group.
We include the extremely short proofs.

\begin{lem}\label{throwawayFrattini}
Let $G$ be a group.
If $X$ is a non-empty subset of $G$ with $G = \langle \Phi(G), X\rangle$, then $G= \langle X\rangle$.
\end{lem}

\begin{proof}
If $\langle X\rangle$ is a proper subgroup of $G$, then there exists a maximal subgroup $M$ of $G$ with $\langle X\rangle \leqslant M$.
Of course, $\Phi(G) \leqslant M$.
Hence, $G = \langle \Phi(G), X\rangle \leqslant M < G$, a contradiction.
\end{proof}

\begin{lem}\label{normal,containedinafrattini}
Let $G$ be a group.
If $K$ is a normal subgroup of $G$ with $K \leqslant \Phi(H)$ for some subgroup $H$ of $G$, then $K \leqslant \Phi(G)$.
\end{lem}

\begin{proof}
If  $K\not\leqslant \Phi(G)$, then $K \not\leqslant M$ for some maximal subgroup $M$ of $G$.
By the maximality of $M$ and the normality of $K$, we deduce that $G = KM$.
By Dedekind's Rule and Lemma \ref{throwawayFrattini}, \[H = H \cap KM = K(H \cap M) =\Phi(H)(H \cap M) = (H \cap M) \leqslant M.\]
Now, $K \leqslant H \leqslant M$, a contradiction.
\end{proof}

\begin{lem}\label{nilpotentimpliesfrattinirespectscontainment}
Let $G$ be a nilpotent group.
If $J \leqslant R \leqslant G$, then $\Phi(J) \leqslant \Phi(R)$.
In other words, in nilpotent groups, taking Frattini subgroups respects containments.
\end{lem}

\begin{proof}
Assume that $\Phi(J) \not\leqslant \Phi(R)$.
Then, there exists a maximal subgroup $S$ of $R$ with $\Phi(J) \not\leqslant S$.
Because $R$ is nilpotent, we know that $S$ is normal in $R$.
We deduce that $R = \Phi(J)S$.
By Dedekind's Rule and Lemma \ref{throwawayFrattini}, \[ J = J\cap R = J \cap  \Phi(J)S= \Phi(J) (J \cap S) = (J \cap S) \leqslant S.\]
Now, $\Phi(J)\leqslant J\leqslant S$, a contradiction.
\end{proof}

Let $G$ be a group, and assume that $\Phi(J) \leqslant \Phi(R)$ whenever $J \leqslant R \leqslant G$.
Nilpotent groups are not the only types of groups that satisfy this condition.
Some structural results concerning the class of groups that satisfy this condition can be found in \cite{kirtland}.

Let $G$ be a group.
Recall the construction of the upper central series of $G$.
Let $\mathbf{Z}_0(G) = 1$, and, for $i>0$, define $\mathbf{Z}_i(G)$ recursively by equation \[ \mathbf{Z}_{i}(G)/\mathbf{Z}_{i-1}(G) = \mathbf{Z}(G/\mathbf{Z}_{i-1}(G)).\]
The ascending series \[ 1 = \mathbf{Z}_0(G) \leqslant \mathbf{Z}_1(G) \leqslant \cdots \] is called the \textit{upper central series} of $G$.
Since we are dealing with finite groups, we know that the upper central series cannot have infinitely many proper containments, and so the upper central series becomes stationary after finitely many term. 
The last term in the upper central series of a group $G$ is called the \textit{hypercenter} of $G$ and is usually denoted by $\mathbf{Z}_\infty(G)$.

A classic exercise is to prove that if $G$ is a group, then $G' \cap \mathbf{Z}(G) \leqslant \Phi(G)$.
See, for example, Exercse 603 in \cite{rosetextbook} or 7.3.26 in \cite{scott}.
It is well-known that this basic fact extends quite nicely.

\begin{lem}\label{G'intersecthypercenterinfrattini}
If $G$ is any group, then \[ G' \cap \mathbf{Z}_\infty(G) \leqslant \Phi(G).\]
\end{lem}

\begin{proof}
Let $J = \Phi(G)$ and, for each non-negative integer $i$, let $D_i = G' \cap \mathbf{Z}_i(G)$.
We shall prove that $D_i \leqslant J$ by induction on $i$.
If $i = 0$, then $D_i = G' \cap \mathbf{Z}_0(G)  = 1 \leqslant J$.

Assume inductively that $D_i \leqslant J$.
Let $M$ be a maximal subgroup of $G$.
If $\mathbf{Z}_{i+1}(G) \leqslant M$, then clearly $D_{i+1} \leqslant M$.
Assume that $\mathbf{Z}_{i+1}(G) \not\leqslant M$.
Then the normality of $\mathbf{Z}_{i+1}(G)$ and the maximality of $M$ imply that $G = \mathbf{Z}_{i+1}(G) M$.
Clearly $M \leqslant \mathbf{N}_G(M)$.
Note that $[M, \mathbf{Z}_{i+1}(G)] \leqslant [G, \mathbf{Z}_{i+1}(G)] \leqslant \mathbf{Z}_i(G)$.
Hence, \[[M, \mathbf{Z}_{i+1}(G)] \leqslant G' \cap \mathbf{Z}_i(G) = D_i \leqslant J \leqslant M.\]
So, $\mathbf{Z}_{i+1}(G) \leqslant \mathbf{N}_G(M)$.
Now, as $G = \mathbf{Z}_{i+1}(G) M$, we have that $M$ is normal in $G$.
By the maximality of $M$ in $G$, the factor group $G/M$ has precisely two subgroups, and so $G/M$ has prime order.
In particular, $G/M$ is abelian, and $D_{i+1} \leqslant G' \leqslant M$.
In all cases, $D_{i+1} \leqslant M$.
Since $M$ was an  arbitrary maximal subgroup of $G$, we deduce that $D_{i+1} \leqslant J$.
The induction argument goes through.
We conclude that $G' \cap \mathbf{Z}_\infty(G) \leqslant J$, as required.
\end{proof}

We need a few lemmas concerning direct products.
The next result appears as Exercise 628 in \cite{rosetextbook} and concerns the Frattini subgroup of a direct product.

\begin{lem}\label{frattinisubgroupofdirectproduct}
Let $G$ be a group and assume that $G = J \times R$.
Then $\Phi(G) = \Phi(J) \times \Phi(R)$.
\end{lem}

The following observation is needed in the proof of Lemma \ref{coresofcertaintypesofsubgroupsindirectproducts}.

\begin{lem}\label{intersections}
Let $G$ be a group and assume that $G  = J \times R$ for subgroups $J$ and $R$ of $G$.
Let $J_1, \ldots, J_n \leqslant J$ and let $R_1,\ldots, R_n \leqslant R$.
Then, \[ \bigcap_{i=1}^n J_iR_i  = \left(\bigcap_{i=1}^n J_i\right)\left(\bigcap_{i=1}^n R_i\right).\]
\end{lem}

\begin{proof}
We use induction on $n$.
The lemma is trivially true for $n = 1$.
Assume inductively the result for $n$.
Then, given subgroups $J_1, \ldots, J_n, J_{n+1} \leqslant J$ and $R_1, \ldots, R_n, R_{n+1} \leqslant R$, we have by our induction hypothesis that \[ \bigcap_{i=1}^{n+1} J_iR_i = \left(\bigcap_{i=1}^n J_iR_i\right) \cap J_{n+1} R_{n+1} = \left(\bigcap_{i=1}^n J_i\right)\left(\bigcap_{i=1}^n R_i\right)  \cap J_{n+1}R_{n+1}.\]
The containment \[  \left(\bigcap_{i=1}^{n+1} J_i\right)\left(\bigcap_{i=1}^{n+1} R_i\right) \leqslant   \left(\bigcap_{i=1}^n J_i\right)\left(\bigcap_{i=1}^n R_i\right)  \cap J_{n+1}R_{n+1}\] is clear.
Let $g \in   \left(\bigcap_{i=1}^n J_i\right)\left(\bigcap_{i=1}^n R_i\right)  \cap J_{n+1}R_{n+1}$.
Now, $g = j_0 r_0 = j_1 r_1$ for $j_0 \in \left(\bigcap_{i=1}^n J_i\right)$, $r_0 \in  \left(\bigcap_{i=1}^n R_i\right)$, $j_1 \in J_{n+1}$, and $r_1 \in R_{n+1}$.
Hence, $j_1^{-1}j_0 = r_1 r_0^{-1} \in J \cap R = 1$.
So, $j_0 = j_1 \in  \left(\bigcap_{i=1}^{n+1} J_i\right)$ and $r_0= r_1 \in \left(\bigcap_{i=1}^{n+1} R_i\right) $.
Indeed, the reverse containment holds too.
\end{proof}

We need the following lemma concerning cores of certain subgroups of direct products.

\begin{lem}\label{coresofcertaintypesofsubgroupsindirectproducts}
Let $G$ be a group and assume that $G = J \times R$ for subgroups $J$ and $R$ of $G$.
Let $V \leqslant J$ and let $W \leqslant R$.
Then \[ (VW)_G = V_J W_R.\]
\end{lem}

\begin{proof}
First, note that $[V,W] \leqslant [J,R] =1$, and so the product set $VW$ indeed forms a subgroup of $G$.
We compute that
\begin{align*}
(VW)_G = \bigcap_{g \in G} (VW)^g 
            = \bigcap_{\substack{j\in J \\ r\in R}} (VW)^{jr} 
             &= \bigcap_{\substack{j\in J \\ r\in R}} V^{jr} W^{jr} \\
             &=\bigcap_{\substack{j\in J \\ r\in R}} V^j W^r \\
            &\overset{(*)}{=}\left( \bigcap_{j \in J} V^j\right)\left( \bigcap_{r \in R} W^r\right)\\
            &= V_J W_R,
\end{align*}
where $(*)$ holds by Lemma \ref{intersections}.
\end{proof}

We now state Rose's Theorem, which appears as Theorem 5 in \cite{rosefrattini}.
As far as we can tell, two proofs of this result exist the literature.
See Theorem 5 in \cite{rosefrattini} and Lemma 14 in \cite{adequatefieldextensions}.
The length and difficulty of the proofs are about the same.
The only ``sophisticated'' result used in these proofs is a classical splitting result due to Gasch\"utz, which we now state here in the prose.
Let $J$ be an abelian normal subgroup of $G$, and let $J \leqslant R \leqslant G$.
Assume that $|J|$ and $|G:R|$ are coprime.
Then, $G$ splits over $J$ if and only if $R$ splits over $J$.
Of course, if $G$ splits over $J$, then a direct application of Dedekind's Rule yields that $R$ splits over $J$.
So, the real content here is that if $R$ splits over $J$, then so does $G$.
This result and its proof can be found in \cite{scott} as Theorem 9.3.5 and in \cite{kurz} as Theorem 3.3.2.

\begin{thrm}[Rose's Theorem]\label{Rose'stheorem}
Let $G$ be a group and let $p$ be a prime.
Let $J$ be a Frattini normal subgroup of $G$ and assume that $J$ is a $p$-group.
Let $S$ be a Sylow $p$-subgroup of $G$.
If $J \leqslant \mathbf{Z}(S)$, then $J \leqslant \Phi(S)$.
\end{thrm}

\section{Properites of the local Frattini subgroup}

We start by verifying that the construction of the local Frattini subgroup of a group does not depend the Sylow subgroups chosen.
This amounts to showing that whenever $P$ and $S$ are Sylow $p$-subgroup of a group $G$, the equality $\Phi(P)_G = \Phi(S)_G$ holds.

\begin{lem}
Let $G$ be a group and let $p$ be a prime.
If $P$ and $S$ are Sylow $p$-subgroups of $G$, then $\Phi(P)_G = \Phi(S)_G$.
\end{lem}

\begin{proof}
Because $\Phi(P)_G$ is a normal $p$-subgroup of $G$, we know that $\Phi(P)_G \leqslant S$.
Assume that $\Phi(P)_G \not\leqslant \Phi(S)$.
Then, there exists a maximal subgroup $R$ of $S$ with $\Phi(P)_G \not\leqslant R$.
Hence, $S = \Phi(P)_GR$.
By Sylow's Theorem, $P = S^g$ for some $g \in G$.
Now, $P = S^g = (\Phi(P)_G R)^g = \Phi(P)_G R^g = \Phi(P) R^g = R^g$ by Lemma \ref{throwawayFrattini}.
So, $|P| = |R^g| = |R| < |S|$, a contradiction.
We deduce that $\Phi(P)_G \leqslant \Phi(S)$ and, since $\Phi(P)_G$ is normal in $G$, we actually have that $\Phi(P)_G \leqslant \Phi(S)_G$.
Reverse the roles of $P$ and $S$ in this argument to see that the reverse containment holds as well.
\end{proof}

\begin{thrm}
Let $G$ be a group.
If $J$ is a normal subgroup of $G$ with $J \leqslant \Psi(G)$, then $J$ is a stictly Frattini subgroup of $G$.
\end{thrm}

\begin{proof}
Let $X$ be an arbitrary subgroup of $G$.
Let $\pi_1 = \pi(G:X)$ and let $\pi_2 = \pi(G: JX)$.
We need to prove that $\pi_1 = \pi_2$.
Since $|G : JX|$ divides $|G : X|$, the containment $\pi_2 \subseteq \pi_1$ certainly holds.

Suppose that $\pi_1 \not\subseteq \pi_2$ and choose $p \in \pi_1 \setminus \pi_2$.
Let $P$ be a Sylow $p$-subgroup of $J$.
Because $J$ is contained in $\Psi(G)$, we know that $J$ is nilpotent.
Hence, $P$ is a characteristic subgroup of $J$.
So, as $J$ is a normal in $G$, we deduce that $P$ is a normal subgroup of $G$.
Now, let $Y$ be a Sylow $p$-subgroup of $X$.
Note that $PY$ is a subgroup of $G$, as $P$ is normal in $G$.
In fact, notice that $PY$ is a Sylow $p$-subgroup of $JX$.
Since $p \not\in \pi_2$, we deduce that $PY$ is a Sylow $p$-subgroup of $G$.

Now, as $PY$ is a Sylow $p$-subgroup of $G$, we know that $\Phi(PY)_G$ is the unique Sylow $p$-subgroup of $\Psi(G)$.
As $P \leqslant J \leqslant \Psi(G)$, we deduce that $P \leqslant \Phi(PY)_G \leqslant \Phi(PY)$ by Sylow's Theorem.
Hence, $PY = \Phi(PY)Y = Y$ by Lemma \ref{throwawayFrattini}.
So $X$ contains a full Sylow $p$-subgroup of $G$, contrary to the fact that $p \in \pi_1$.
Hence, $\pi_1 \subseteq \pi_2$.
Combining our containments, we conlude that $\pi_1 = \pi_2$, as wanted.
\end{proof}

In particular, notice that for each Sylow subgroup $S$ of $G$, the subgroup $\Phi(S)_G$ is a strictly Frattini subgroup of $G$.

In Section 2, we saw that for each group $G$, the containment $G' \cap \mathbf{Z}_\infty(G) \leqslant \Phi(G)$ holds.
We now show that the intersection $G' \cap \mathbf{Z}_\infty(G)$ lies deeper in $\Phi(G)$.

\begin{thrm}\label{G'intersecthypercenter}
If $G$ is any group, then \[ G' \cap \mathbf{Z}_{\infty}(G) \leqslant \Psi(G).\]
\end{thrm}

\begin{proof}
Let $J = \Psi(G)$ and let $Z_i  =  \mathbf{Z}_i(G)$ for each non-negative integer $i$.
Let $D_i = G' \cap Z_i$ for each non-negative integer $i$.
We shall show that $D_i \leqslant J$ by induction on $i$.
If $i = 0$, then $D_i = 1 \leqslant  J$.

Assume inductively that $D_i \leqslant J$.
Let $p$ be a prime and let $P$ be a Sylow $p$-subgroup of $D_{i+1}$.
Because $D_{i+1} \leqslant Z_{i+1}$, we know that $D_{i+1}$ is nilpotent.
Hence, $P$ is a characteristic subgroup of $D_{i+1}$.
Since $D_{i+1}$ is normal in $G$, we deduce that $P$ is a normal subgroup of $G$.
Note that $D_{i+1} \leqslant \Phi(G)$ by Lemma \ref{G'intersecthypercenterinfrattini}, and so $P \leqslant \Phi(G)$.
In particular, $P$ is a Frattini normal subgroup of $G$.

Let $R = [G,P]$.
Note that $R \leqslant Z_i$ as $P \leqslant Z_{i+1}$.
So, $R \leqslant D_i \leqslant J$ by our induction hypothesis.
Let $S$ be a Sylow $p$-subgroup of $G$.
Hence, $\Phi(S)_G$ is the Sylow $p$-subgroup of $J$.
Now, $R \leqslant \Phi(S)_G \leqslant \Phi(S)$.
Also, note that $R$ is a normal subgroup of $G$.
(See, for example, Exercise 179 in \cite{rosetextbook}.)
By Lemma \ref{normal,containedinafrattini}, $R \leqslant \Phi(G)$.

Consider $G/R$.
Note that $J/R$ is a normal subgroup of $G$, and note also that $J/R \leqslant \Phi(G)R/R = \Phi(G/R)$.
Hence, $J/R$ is a Frattini normal subgroup of $G/R$.
The subgroup $S/R$ is a Sylow $p$-subgroup of $G/R$.
Since $[S/R, J/R] =[S,J]R/R \leqslant R/R$, we have that $J/R \leqslant \mathbf{Z}(S/R)$.
By Rose's Theorem (Theorem \ref{Rose'stheorem}), $J/R \leqslant \Phi(S/R)$.
As $R \leqslant \Phi(S)$, we have that $\Phi(S/R) = \Phi(S)/R$, and so $P \leqslant \Phi(S)$.
Now, since $P$ is a normal subgroup of $G$, we actually have that $P \leqslant \Phi(S)_G \leqslant J$.
Since $p$ was an arbitrary prime, this argument shows that each Sylow subgroup of $D_{i+1}$ is a subgroup of $J$.
Consequently, $D_{i+1} \leqslant J$.
The induction argument goes through.
In particular, $G' \cap \mathbf{Z}_{\infty}(G) \leqslant J$, as required.
\end{proof}

We now prove that $\mathbf{F}(G)/\Psi(G)$ is abelian for any group $G$.

\begin{thrm}\label{F/Psiabelian}
If $G$ is any group, then the section $\mathbf{F}(G)/\Psi(G)$ is abelian.
\end{thrm}

\begin{proof}
Let $F = \mathbf{F}(G)$ and let $J = \Psi(G)$.
Let $p$ be a prime and let $P$ be a Sylow $p$-subgroup of $F$.
Of course, as $F$ is nilpotent, we know that $P$ is a characteristic subgroup of $F$.
Since $F$ is normal in $G$, we deduce that $P$ is a normal subgroup of $G$.
Now, since $P'$ is a characteristic subgroup of $P$ and $P$ is normal in $G$, the subgroup $P'$ is normal in $G$.

Let $S$ be a Sylow $p$-subgroup of $G$.
Note that $P' \leqslant S$.
Because $P$ is nilpotent, $P' \leqslant \Phi(P)$.
Since $S$ is nilpotent, $\Phi(P) \leqslant \Phi(S)$ by Lemma \ref{nilpotentimpliesfrattinirespectscontainment}.
So, $P' \leqslant \Phi(S)$.
In fact, since $P'$ is normal in $G$, we have that $P' \leqslant \Phi(S)_G \leqslant J$.

Now, $PJ/J$ is a Sylow $p$-subgroup of $F/J$.
Also, we have that $[PJ/J, PJ/J]=P'J/J = J/J$.
Hence, $PJ/J$ is abelian.
Since $p$ was arbitrary, the factor group $F/J$ has abelian Sylow subgroups.
The group $F/J$ is, therefore, nilpotent with abelian Sylow subgroups.
Hence, $F/J$ is abelian, and $F' \leqslant J$.
Indeed, the group $F/J$ is abelian.
\end{proof}

\begin{cor}
If $G$ is any group, then the section $\mathbf{F}(G)/\Phi(G)$ is abelian.
\end{cor}

\begin{proof}
Note that $\mathbf{F}(G)' \leqslant \Psi(G)$ by Theorem \ref{F/Psiabelian}.
Hence $\mathbf{F}(G)' \leqslant \Phi(G)$, and so $\mathbf{F}(G)/\Phi(G)$ is abelian.
\end{proof}

As mentioned, if $G = J \times R$, then $\Phi(G) = \Phi(J) \times \Phi(R)$.
We now prove the analogous result for the local Frattini subgroup.

\begin{thrm}
Let $G$ be a group and let $G = J\times R$ for subgroups $J$ and $R$ of $G$.
Then $\Psi(G) = \Psi(J) \times \Phi(R)$.
\end{thrm}

\begin{proof}
Let $p$ be a prime.
Let $J_p$ be a Sylow $p$-subgroup of $J$ and let $R_p$ be a Sylow $p$-subgroup of $R$.
Note that $S_p = J_p R_p = J_p \times R_p$ is a Sylow $p$-subgroup of $G$.
By Lemma \ref{frattinisubgroupofdirectproduct}, $\Phi(S_p) = \Phi(J_p) \times \Phi(R_p)$.
Now, by Lemma \ref{coresofcertaintypesofsubgroupsindirectproducts}, $\Phi(S_p)_G = (\Phi(J_p) \Phi(R_p))_G = \Phi(J_p)_J\Phi(R_p)_R$.
Since $p$ was an arbitrary prime, we see 
\begin{align*}
\Psi(G) = \prod_{p \in \pi(G)} \Phi(S_p)_G &= \prod_{p\in \pi(G)} \Phi(J_p)_J \Phi(R_p)_R\\ &=\left(\prod_{p \in \pi(J)} \Phi(J_p)_J\right)\left(\prod_{p \in \pi(R)} \Phi(R_p)_R\right) \\
&= \Psi(J) \times \Psi(R),
\end{align*}
as wanted.
\end{proof}

\section{More variations on Rose's Theorem}

We recall the situation of Rose's Theorem.
Let $G$ be a group and let $p$ be a prime.
Let $J$ be a normal $p$-subgroup of $G$ with $J \leqslant \Phi(G)$.
Let $S$ be a Sylow $p$-subgroup of $G$.
J.S. Rose prove that if $J$ is a central subgroup of $S$, then, in fact, $J$ is a Frattini normal subgroup of $S$.
As mentioned above, more conditions that would imply $J \leqslant \Phi(S)$ were proved in \cite{adequatefieldextensions}.
For example, the authors prove that if $J$ is cyclic or if the subgroup $J \cap S'$ happens to be normal in $G$, then $J \leqslant \Phi(S)$.

We first prove that if a certain commutator subgroup is normal in the whole group $G$, then $J \leqslant \Phi(S)$.

\begin{thrm}
Let $G$ be a group and let $p$ be a prime.
Let $J$ be a normal $p$-subgroup of $G$ with $J \leqslant \Phi(G)$.
Let $S$ be a Sylow $p$-subgroup of $G$.
If the commutator subgroup $[S,J]$ is normal in $G$, then $J \leqslant \Phi(S)$.
So, in particular, $J \leqslant \Phi(S)_G \leqslant \Psi(G)$.
\end{thrm}

\begin{proof}
Let $R = [S,J]$.
Then, by our hypothesis, $R$ is a normal subgroup of $G$.
Of course, $R \leqslant J$, as $J$ is normal in $G$ (and, therefore, $S$).
Consider the group $G/R$.
Note that $J/R$ is a normal $p$-subgroup of $G/R$ with $J/R \leqslant \Phi(G)R/R \leqslant \Phi(G/R)$.
Also, by a basic commutator identity, $[S/R, J/R]=[S,J]R/R= R/R$.
Hence, $J/R \leqslant \mathbf{Z}(S/R)$.
By Rose's Theorem applied in $G/R$, we deduce that $J/R \leqslant \Phi(S/R)$.
Note that $R \leqslant S'$.
Since $S$ is nilpotent, $S' \leqslant \Phi(S)$.
Hence, $\Phi(S/R) = \Phi(S)/R$ and $J\leqslant \Phi(S)$, as wanted.
\end{proof}

The next theorem shows that if $J$ is a Frattini normal $p$-subgroup of $G$ that happens to be contained in the hypercenter of $G$, then $J$ is contained in $\Phi(S)$ for any Sylow $p$-subgroup $S$ of $G$.

\begin{thrm}
Let $G$ be a group and let $p$ be a prime.
Let $J$ be a normal $p$-subgroup of $G$ with $J \leqslant \Phi(G)$.
Let $S$ be a Sylow $p$-subgroup of $G$.
If $J \leqslant \mathbf{Z}_\infty(G)$, then $J \leqslant \Phi(S)$.
So, in particular, $J \leqslant \Phi(S)_G \leqslant \Psi(G)$.
\end{thrm}

\begin{proof}
Let $H = \mathbf{Z}_\infty(G)$.
Let $R = [G, J]$, and note that $R$ is a normal subgroup of $\langle G, J \rangle = G$.
Now, $R \leqslant (G' \cap H) \leqslant \Psi(G)$ by Theorem \ref{G'intersecthypercenter}.
Since $R$ is a subgroup of $J$, we know that $R$ is a $p$-subgroup of $G$.
So, $R \leqslant (S \cap \Psi(G)) = \Phi(S)_G \leqslant \Phi(S)$.
Observe that $J/R$ is a normal subgroup of $G/R$ and that $J/R \leqslant \Phi(G)R/R \leqslant \Phi(G/R)$.
Also $[S/R, J/R] =[S,J]R/R = R/R$, and so $J/R \leqslant \mathbf{Z}(S/R)$.
By Rose's Theorem, $J/R \leqslant \Phi(S/R)$.
Now, since $R \leqslant \Phi(S)$, we know that $\Phi(S/R)  = \Phi(S)/R$.
Hence, $J \leqslant \Phi(S)$, as wanted.
\end{proof} 

We can now state the known results in the following omnibus theorem, which includes Rose's Theorem, our theorems above, and Theorem 9 in \cite{adequatefieldextensions}.
In order to make the statement cleaner, we let $\mathfrak{G}$ be the group constructed in Example 9 in \cite{rosefrattini}, which is a particular split extension of an elementary abelian $3$-group of order $3^4$ by $SL(2,3)$.

\begin{thrm}
Let $G$ be a group and let $p$ be a prime.
Let $J$ be a normal $p$-subgroup of $G$ with $J \leqslant \Phi(G)$.
Let $S$ be a Sylow $p$-subgroup of $G$.
\begin{enumerate}
\item If $J \leqslant \mathbf{Z}(S)$, then $J \leqslant \Phi(S)$.
\item If $J$ is cyclic, then $J \leqslant \Phi(S)$.
\item If the subgroup $J \cap S'$ is normal in $G$, then $J \leqslant \Phi(S)$.
\item If the subgroup $[S,J]$ is normal in $G$, then $J \leqslant \Phi(S)$.
\item If $J \leqslant \mathbf{Z}_\infty(G)$, then $J \leqslant \Phi(S)$.
\item If $G$ is $p$-solvable of $p$-length $1$, then $J \leqslant \Phi(S)$.
\item If $G$ is $p$-solvable and each subgroup of $S$ can be generated by at most four elements and if $p =3$ and $\mathfrak{G}$ is not involved in $G$, then $J \leqslant \Phi(S)$.
\item If $S$ is non-cyclic with $|S : \Phi(S)| = p^2$ and if either (a) $p = \min(\pi(G))$, (b) $S$ is generated by the elements of order $p$, or (c) $G$ is $p$-solvable, then $J \leqslant \Phi(S)$.
\end{enumerate}
\end{thrm}

\end{document}